\documentclass[10pt]{amsart}

\title[Long systolic pants decompositions]{Hyperbolic surfaces with long systoles\\ that form a pants decomposition}
\author{Bram Petri}
\address{Max Planck Institute for Mathematics, Bonn, Germany}
\email{brampetri@mpim-bonn.mpg.de}
\date{\today}

\newtheorem{thm}{Theorem}[section]
\newtheorem{prp}[thm]{Proposition}
\newtheorem{cor}[thm]{Corollary}
\newtheorem{lem}[thm]{Lemma}

\theoremstyle{definition}

\usepackage{amsmath,amsfonts,amssymb,graphicx,overpic}
\usepackage{float,enumitem}
\usepackage[enableskew]{youngtab}

\newcommand{\st}[2]{\left\{ #1 ;\; #2\right\}}

\newcommand{\card}[1]{\left| #1 \right|}

\newcommand{\sys}{\mathrm{sys}}

\newcommand{\diam}{\mathrm{diam}}
\newcommand{\dist}{\mathrm{d}}

\newcommand{\NN}{\mathbb{N}}

\begin{document}

\begin{abstract} We present a construction of sequences of closed hyperbolic surfaces that have long systoles which form pants decompositions of these surfaces. The length of the systoles of these surfaces grows logarithmically as a function of their genus.
\end{abstract}

\maketitle

\section{Introduction}

A systole of a closed hyperbolic surface is a shortest non-contractible curve on this surface. In this note, we will use the word systole both for the curve itself and for its length. 

The systole function achieves a maximum among all closed hyperbolic surfaces of a given genus. This follows from the compactness of the thick part of the moduli space of these surfaces. What this maximum should be is wide open in general. First of all, if a closed hyperbolic surface $X$ has genus $g(X)$, then a simple area argument shows that its systole $\sys(X)$ satisfies
\[\sys(X) \leq 2\log(4g(X)-2) \]
(see for instance \cite[Lemma 5.2.1]{Bus2}. Because of this inequality, a sequence of surfaces with \emph{long systoles} will be a sequence of compact hyperbolic surfaces $\{X_k\}_{k\in\NN}$ with genus $g(X_k)\to\infty$ as $k\to \infty$ and systoles $\sys(X_k)\geq C\log(g(X_k))$ for some constant $C>0$ independent of $k$.

It is known that surfaces with long systoles exist. In particular, in \cite{BS} Buser and Sarnak show that there are sequences $\{X_k\}_k$ of congruence covers of a fixed arithmetic surface such that
\[g(X_k)\to\infty\;\text{as }k\to\infty\;\;\text{and}\;\;\lim_{k\to\infty}\frac{\sys(X_k)}{\log(g(X_k))} \geq \frac{4}{3}.\]
This construction was later generalized to a larger class of surfaces by Katz, Schaps and Vishne in \cite{KSV}. \cite{PW} contains another construction of surfaces with long systoles, based on gluing hyperbolic triangles together. In the latter case the multiplicative constant in front of the logarithm equals $1$. In the case of cusped hyperbolic surfaces, there are similar bounds and constructions, for which we refer to \cite{Sch1}, \cite{Sch2}, \cite{Par2}, \cite{FP}.

The question we ask in this paper is a slight variation of the above problem. We not only ask for long systoles, but also ask that these systoles form a pants decomposition of the surface. Recall that a pants decomposition of a surface is a set of simple closed curves that decompose it into three-holed spheres (pairs of pants).

Let us first note that it isn't hard to construct a hyperbolic surface of which the systoles form a pants decomposition: take any pants decomposition and pinch all the curves down to some length $\varepsilon$. If $\varepsilon$ is small enough, the collar lemma guarantees that these curves are indeed systoles (see eg. \cite[Chapter 4]{Bus2} for details). This doesn't answer our question, because for this argument to work, $\varepsilon$ needs to be uniformly small.

There also already exists a construction of closed hyperbolic surfaces with growing systoles that form a pants decomposition. Namely, in \cite{Bus1}, Buser constructed surfaces with systolic pants decompositions consisting of curves of length roughly $\sqrt{\log(g)}$.

The main result (Theorem \ref{thm_pantssystoles}) of this article is that there exist sequences of closed hyperbolic surfaces $\{X_n\}_{n\in\NN}$ of which the systoles form a pants decomposition and
\[\sys(X_n)\geq \frac{4}{7}\log (g(X_n)) - K \;\;\text{and}\;\;g(X_n)\to\infty,\]
as $n\to\infty$, where $K$ is a constant independent of $n$.

The largest part of our proof, ending with Proposition \ref{prp_existence}, consists of finding a hyperbolic surface which for every $a\in (0,\infty)$ large enough:
\begin{itemize}[leftmargin=0.2in]
\item[1.] has a pants decomposition of curves that all have the same length $a$ and
\item[2.] all simple curves that are not part of the given pants decomposition are `long enough'.
\end{itemize}
The genus of this surface might much larger than what we want. However, if this is the case, we can modify it so that the genus achieves the bound we claim. This relies on a comparison between a type of diameter of the surface and its systole. A similar idea has been used by Erd\H{o}s and Sachs to construct regular graphs with large girth (the length of a shortest cycle in this graph) \cite{ES}.

\section*{Acknowledgment}

The author would like to thank Federica Fanoni, Ursula Hamenst\"adt, Nati Linial and Hugo Parlier for useful conversations. He is also grateful to the Max Planck Institute for Mathematics in Bonn for providing a pleasant work environment while this project was carried out.

\section{Background}

Recall that a closed hyperbolic surface is a compact $2$-manifold without boundary, equipped with a complete Riemannian metric of constant curvature $-1$. For the geometry of hyperbolic surfaces in general, we refer to \cite{Bus2}.

Given $a\in(0,\infty)$, let $P_a$ denote the hyperbolic pair of pants (three-holed sphere) of which all the boundary components have length $a$. Recall that this defines $P_a$ up to isometry.

To obtain surfaces with systolic pants decompositions and long systoles, we need to find a fortunate way to glue copies of $P_a$ together into a closed hyperbolic surface. Of course, if we want to have any chance of obtaining surfaces with large systoles, we need to choose $a$ large. Unfortunately, this will make the seams of $P_a$ (the shortest geodesic segments between all pairs of distinct boundary components of $P_a$) very short. Hence, the ultimate goal will be to balance these two effects.

Before we get to this (Section \ref{sec_construction}), we will first need to recall and prove some facts about the geometry of $P_a$ (Sections \ref{sec_pairofpants} and \ref{sec_arcs}) and spheres with boundary glued out of multiple copies of $P_a$ (Section \ref{sec_trees}).

\section{Pairs of pants}\label{sec_pairofpants}

Let us first gather some facts about the geometry of $P_a$. We will write $s(a)$ for the length of the three seams on $P_a$. Using the formulas on \cite[Page 454]{Bus2}, we obtain
\[\cosh(s(a)) = \frac{\cosh^2\left(\frac{a}{2} \right)+\cosh\left(\frac{a}{2}\right)}{\sinh^2\left(\frac{a}{2}\right)} = \frac{\cosh\left(\frac{a}{2}\right)}{\cosh\left(\frac{a}{2}\right) - 1}.\]
Working this out we obtain
\[s(a)\sim 2e^{-a/4}, \]
as $a\to\infty$, where we write $f(a)\sim g(a)$ as $a\to \infty$ if
\[\lim_{a\to\infty} \frac{f(a)}{g(a)}=1. \]

Every boundary component of $P_a$ has two pairs of special points: the points where the seams meet the boundary and the midpoints between each pair of seams, indicated in white and black respectively in the figure below:
\begin{figure}[H]
\begin{center}
\includegraphics[scale=1]{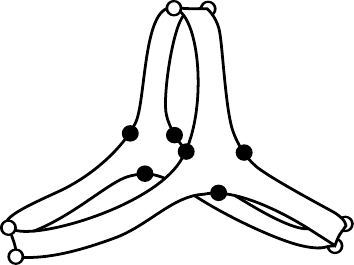} 
\end{center}
\caption{The pair of pants $P_a$.}
\label{pic_pants}
\end{figure}

All gluings of copies of $P_a$ into hyperbolic surfaces that we shall consider will be such that white points are glued to black points and vice versa. Note that this comes down to twists of either $a/4$ or $3a/4$. We will call such gluings \emph{admissible gluings}. 

For the arguments later on, we will need a bound on the diameter $\diam(P_a)$ of $P_a$ as a hyperbolic surface. To this end we have the following lemma: 
\begin{lem}\label{lem_diamP}
There exists a constant $C>0$, independent of $a$, so that
\[ \frac{a}{2}\leq \diam(P_a) \leq \frac{a}{2}+C,\]
for all $a$ large enough. 
\end{lem}
\begin{proof} The lower bound is for instance achieved by any pair midpoints on a single boundary component.

For the upper bound, we note that for $a$ large enough, the distance between any point on $P_a$ and $\partial P_a$ is uniformly bounded. As such we can instead bound $\max\st{\dist_{P_a}(p,q)}{p,q\in\partial P_a}$, where $\dist_X:X\times X\to [0,\infty)$ denotes the distance function on a hyperbolic surface $X$. 

Let $p,q\in\partial P_a$. Their distance is bounded from above by $a/2$ if they lie on the same boundary component, so let us assume that they do not. Denote by $m_{p,1}$, $m_{p,2}$ and $m_{q,1}$, $m_{q,2}$ the midpoints on the boundary components containing $p$ and $q$ respectively. Furthermore, we may label these midpoints so that when we cut $P_a$ open along the seams, $m_{p,1}$ and $m_{q,1}$ lie on the same right-angled hexagon and $m_{p,2}$ and $m_{q,2}$ do as well. As such 
\[\dist(m_{p,1},m_{q,1}) \leq D \;\;\text{and}\;\; \dist(m_{p,2},m_{q,2}) \leq D \]
for some $D>0$ that does not depend on $a$ (given that $a$ is large enough). The formulas on \cite[Page 454]{Bus2} can be used to work out an explicit bound for this constant.

We have
\[\dist_{P_a}(p,m_{p,1})+\dist_{P_a}(p,m_{p,2})+\dist_{P_a}(q,m_{q,1})+\dist_{P_a}(q,m_{q,2})=a. \]
So either
\[\dist_{P_a}(p,m_{p,1})+\dist_{P_a}(q,m_{q,1})\leq \frac{a}{2}\;\;\text{or}\;\;\dist_{P_a}(p,m_{p,2})+\dist_{P_a}(q,m_{q,2})\leq \frac{a}{2}.\]
Together with the bound on the distances between midpoints, this implies the lemma.
\end{proof}

For convenience, let us define a constant $a_0\in (0,\infty)$ so that 
\begin{equation}\label{eq_defa0}
s(a) \leq 3a/16 \;\;\text{and}\;\; 2\cdot \diam(P_a)/a \leq \min\{3/2-8s(a)/a,\;1+C/a\},
\end{equation}
for all $a\geq a_0$. The use of these inequalities will become apparent later on.

\section{Lengths of simple arcs}\label{sec_arcs}

In this section we record two lemmas about the lengths of arcs on $P_a$. The first, in which $\dist_X:X\times X\to [0,\infty)$ denotes the distance function on a hyperbolic surface $X$, is as follows:

\begin{lem}\label{lem_arcs} Let $\gamma \subset P_a$ be a simple arc between two points $p,q\in\partial P_a$ in the same boundary component and suppose that $\gamma$ is not homtopic to a segment in $\partial P_a$ relative to $p$ and $q$. Let $s_p$ and $s_q$ be the seams nearest to $p$ and $q$ respectively and write $x=\dist_{P_a}(p,s_p)$ and $y=\dist_{P_a}(q,s_q)$. Then
\[\ell(\gamma) \geq a - 4s(a)-x-y.\]
\end{lem}

\begin{proof} Because $\gamma$ is not homotopic to a segment in $\partial P_a$, it necessarily intersects one of the seams. We will prove our lemma in three steps: first for arcs that intersect the seams on $P_a$ only once, then for arcs that intersect the seams twice and finally for arcs that intersect the seams at least three times.

Consider the following figure for the first case:
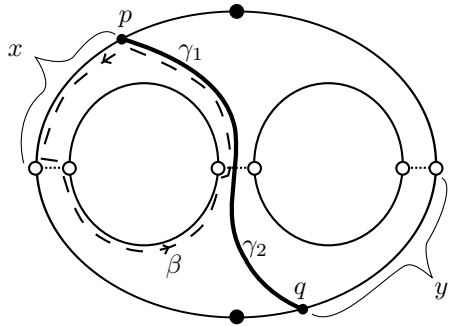
\begin{figure}[H]
\begin{center}
\begin{overpic}[scale=1]{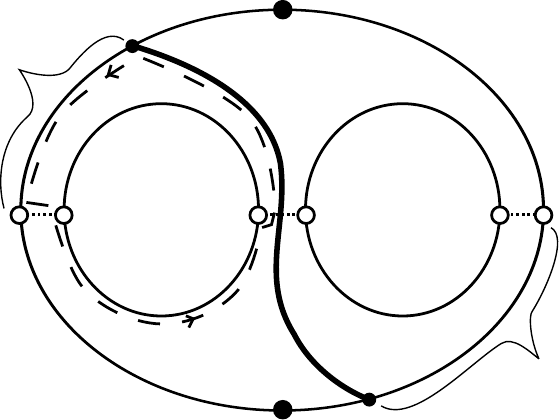} 
\put (23,71) {$p$}
\put (64,7) {$q$}
\put (-3,63) {$x$}
\put (97,7) {$y$}
\put (37,63) {$\gamma_{1}$}
\put (52,17) {$\gamma_{2}$}
\put (34,12) {$\beta$}
\end{overpic}
\end{center}
\caption{An arc on $P_a$ that intersects the seams once.}
\label{pic_pantsarc1}
\end{figure}

It shows the arc $\gamma$ that is cut into two sub-arcs $\gamma_{1}$ and $\gamma_{2}$ by one of the seams. Note that $\gamma$ necessarily intersects the `middle' seam, if not it would be homotopic to the boundary.

The dashed curve $\beta$ in the image is homotopic to one of the boundary curves of $P_a$. As such
\[ \ell(\gamma_{1})+2s(a)+x + \frac{a}{2} \geq \ell(\beta)\geq a \; \Rightarrow \; \ell(\gamma_{1}) \geq \frac{a}{2}-x-2s(a).\]
Likewise, we have
\[\ell(\gamma_{2}) \geq \frac{a}{2}-y-2s(a). \]
Adding these up gives us the inequality we are after. Because the case where the nearest seam to $p$ and $q$ is the same is completely analogous, this completes the proof in the case where $\gamma$ intersects the seems once.

The case where $\gamma$ intersects the seams twice splits into two situations. First note that $\gamma$ necessarily intersects two distinct seams, otherwise it would be homotopic to a boundary segment. 

We start with assumption that the seams closest to $p$ and $q$ are the same. Suppose that of the two, $q$ is closest to the seam, in this case, the situation looks as follows:
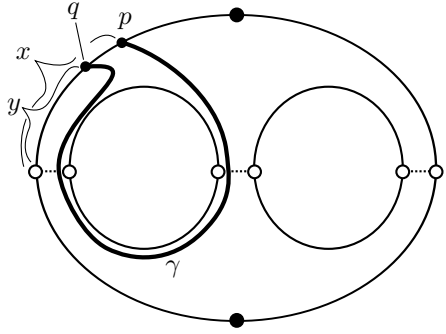
\begin{figure}[H]
\begin{center}
\begin{overpic}[scale=1]{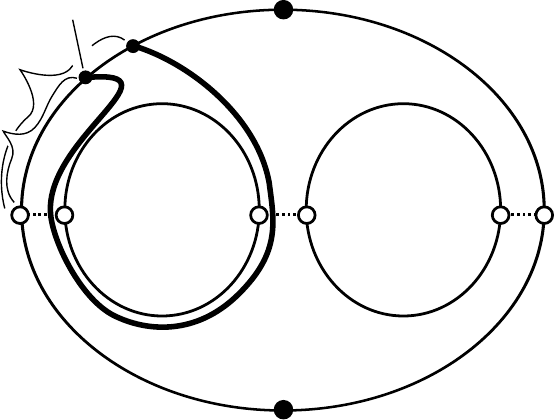} 
\put (23,71) {$p$}
\put (11,73.5) {$q$}
\put (-1,63) {$x$}
\put (-3,51) {$y$}
\put (34,13) {$\gamma$}
\end{overpic}
\end{center}
\caption{An arc on $P_a$ that intersects the seams twice.}
\label{pic_pantsarc2}
\end{figure}

We can complete $\gamma$ with a boundary segment into a curve homotopic to a boundary curve. As such we obtain
\[\ell(\gamma) \geq a - x + y \geq a - 4s(a) -x - y.\]
If $p$ is closest to the seam instead then the plus and minus signs in front of $x$ and $y$ interchange, which means that the lemma is still valid.

The second possibility is that the seam closest to $p$ is not the seam closest to $q$. In that case, the situation looks as in the following figure:
\begin{figure}[H]
\begin{center}
\begin{overpic}[scale=1]{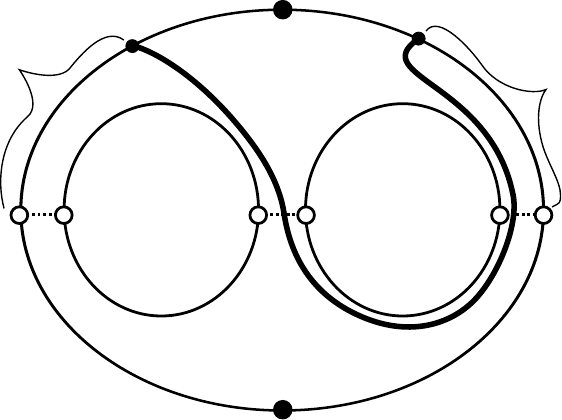} 
\put (23,71) {$p$}
\put (74,72) {$q$}
\put (-3,63) {$x$}
\put (98,60) {$y$}
\put (36,60) {$\gamma_{1}$}
\put (64,13) {$\gamma_{2}$}
\end{overpic}
\end{center}
\caption{Another arc on $P_a$ that intersects the seams twice.}
\label{pic_pantsarc3}
\end{figure}
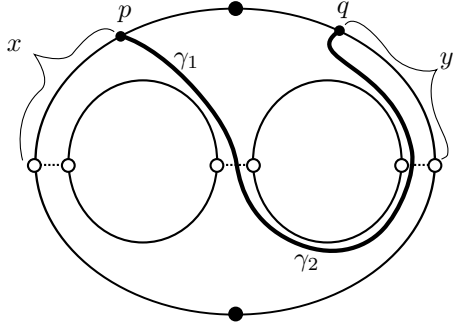
With similar arguments to the one above, we obtain
\[\ell(\gamma_{1}) \geq \frac{a}{2}-x-2s(a)\;\;\text{and}\;\;\ell(\gamma_{2})\geq \frac{a}{2}-2s(a), \]
hence
\[\ell(\gamma) \geq a - x -4s(a) \geq a - x - y -4s(a). \]
Again, we might change which pair of seams $\gamma$ intersects, which changes the role of $x$ and $y$ in the above, but not the validity of the lemma.

Finally, in the case where $\gamma$ intersects the seams at least three times, $\gamma$ has at least two sub-arcs that run from seam to seam. By the same arguments as the ones above, such an arc has length at least $\frac{a}{2}-2s(a)$, from which we obtain that
\[\ell(\gamma) \geq a - 4s(a), \]
which finishes the proof.
\end{proof}

Furthermore, we have the following:
\begin{lem}\label{lem_shortcuts} Let $\alpha\subset P_a$ be a simple arc between two points $p$ and $q$ on a single boundary component of $P_a$. Furthermore, let $\beta_1$ and $\beta_2$ be the two arcs between $p$ and $q$ on the given boundary component. Then
\[\ell(\alpha) \geq \max\{\ell(\beta_1),\ell(\beta_2)\}. \]
\end{lem}

\begin{proof} 
Set $b=\ell(\beta_1)$. Hence $\ell(\beta_2)=a-b$. Because $\alpha\cup\beta_1$ is homotopic to one of the pants curves, we have
\[\ell(\alpha) + b \geq a \; \Rightarrow \; \ell(\alpha) \geq a-b = \ell(\beta_2).\]
Likewise, $\alpha\cup\beta_2$ is homotopic to the other pants curve, so
\[\ell(\alpha)+a-b \geq a \; \Rightarrow \; \ell(\alpha) \geq b=\ell(\beta_1).\]
\end{proof}

\section{Trees of pants}\label{sec_trees}

The two lemmas above allow us to control the lengths of simple curves on spheres with boundary obtained from admissible gluings of multiple copies of $P_a$. In this section we detail how to do this.

Given $k\in\NN$, let $B_k$ be the radius $k$ ball around a vertex in the infinite $3$-regular tree. So $B_k$ is a tree in which all the vertices except the leaves have degree $3$. Furthermore, let $X_k$ be the surface obtained by gluing copies of $P_a$ together in an admissible way according to $B_k$. Note that this uniquely defines the surface $X_k$, this follows from the symmetry of $B_k$. Topologically $X_k$ is a sphere with $3\cdot 2^k$ boundary components.

We have the following corollary to the lemmas above:
\begin{cor}\label{cor_nonpantscurves} Let $k\in \NN$ and $a\geq a_0$. Furthermore, let $\gamma$ be a simple closed curve on $X_k$ that is not homotopic to a pants curve. Then
\[\ell(\gamma) \geq 3a/2 - 8s(a). \]
\end{cor}

\begin{proof} We may and will assume that $\gamma$ is a simple closed geodesic. Any such curve $\gamma$ intersects the pants curves in a positive even number of points. In particular, $\gamma$ crosses at least two pairs of pants.

Let us denote the points where $\gamma$ intersects the pants curves $p_1,\ldots,p_{2m}$. Furthermore, let $\gamma_i$ be the sub-arc of $\gamma$ between $p_i$ and $p_{i+1}$ for $i=1,\ldots,2m$, where we set $p_{2m+1}=p_1$. Note that because geodesics do not form bigons, none of these $\gamma_i$ are homotopic to a segment of a pants curve relative to $p_i$ and $p_{i+1}$. Let $P_i$ be the pair of pants on which $\gamma_i$ lies and let $s_{i,1}$ and $s_{i,2}$ denote the seams of $P_i$ that are nearest to $p_i$ and $p_{i+1}$ respectively and write
\[x_i =  \dist_{P_i}(p_i,s_{i,1}) \;\;\text{and}\;\; y_i =  \dist_{P_i}(p_{i+1},s_{i,2})  \]
Because the gluing is admissible, we have
\begin{equation}\label{eq_admiss}
x_{i+1} = \frac{a}{4}-y_i.
\end{equation}
Furthermore, because every pants curve on $X_k$ is separating, at least two of the segments $\gamma_i$ run between two points on a single pants curve.

Let us first assume that $\gamma$ lies on two pairs of pants only. In particular, this means that all the points $p_i$ lie on the same pants curve. Lemma \ref{lem_arcs} tells us that
\[
\ell(\gamma) = \sum_{i=1}^{2m} \ell(\gamma_i) \geq \sum_{i=1}^{2m} a-x_i-y_i - 4s(a) = 2ma - 8ms(a) - \left(\sum_{i=1}^{2m}x_i+y_i\right).\]
Now we apply \eqref{eq_admiss} and obtain
\[\ell(\gamma) \geq m\cdot \left(\frac{3}{2}a - 8s(a)\right). \]
Because $a\geq a_0$, the right hand side is minimized when $m=1$ (see \eqref{eq_defa0}).

Our goal will be to use Lemma \ref{lem_shortcuts} to show that if $\gamma$ runs through $n\geq 3$ pairs of pants, then we can modify it, so that it becomes a curve that is shorter, lies in $n-1$ pairs of pants and is still not a pants curve. Note that this is enough to prove the corollary.

The pairs of pants that contain $\gamma$ induce a sub-tree $T_\gamma\subset B_k$ and a subsurface $X_\gamma\subset X_k$. Consider a pair of pants $P$ that corresponds to a leaf of $T_\gamma$. Denote the pants curves of $P$ that do not intersect $\gamma$ by $\alpha_1$ and $\alpha_2$. When restricted to $P$, $\gamma$ necessarily consists of a disjoint union of arcs running between pairs of points on a single boundary component of $P$. Because each of these arcs is separating on $P$, we can speak of the arcs closest to $\alpha_1$ and $\alpha_2$ respectively. These are the arcs $\gamma_1$ and $\gamma_2$ such that one component of $P\setminus\gamma_i$ contains $\alpha_i$ and no other arc of $\gamma$. 

Because $T_\gamma$ consists of more than $3$ pairs of pants, it has at least $5$ boundary components. Furthermore, since $\gamma$ is simple and separating, $T_\gamma\setminus\gamma$ consists of two components. Hence, there is a component of $T_\gamma\setminus\gamma$ that has at least $3$ boundary components (excluding the boundary component coming from $\gamma$). Exactly one of these is $\alpha_1$ or $\alpha_2$. Without loss of generality, suppose it is $\alpha_1$. 

We may now replace $\gamma_1$ by one of the two segments between the two points where it meets $\partial P$. The segment we choose is the segment on the side of $\alpha_1$. By construction, the modified curve is still simple and still wraps around at least two boundary components of $X_k$ and hence is still not homotopic to a pants curve. The segment we added can however be homotoped into $X_\gamma\setminus P$. By Lemma \ref{lem_shortcuts} the resulting curve is also shorter than $\gamma$. 

As such we can remove the arcs of $\gamma$ in $P$ one by one until the resulting curve can be homotoped to lie in $X_\gamma\setminus P$ and is shorter than $\gamma$.
\end{proof}

\section{The construction}\label{sec_construction}

We are now ready to describe the construction. First of all, given $a>0$, a trivalent graph $\Gamma$ with vertices $V$ and edges $E$ and $t\in \{-1,1\}^E$, we define a hyperbolic surface $S_a(\Gamma,t)$. $S_a(\Gamma,t)$ is the surface obtained by gluing copies of $P_a$ according to $\Gamma$ with twist $\tau_e$ along the edge $e$ for all $e\in E$, where
\[\tau_e = \left\{ \begin{array}{ll} a/4 & \text{if }t_e=-1 \\ 3a/4 & \text{if }t_e=1 \end{array}\right. \text{ for all } e\in E.\]
Here we measure twist as the distance between the images of some predefined triple of midpoints (the black points in Figure \ref{pic_pants}) on $\partial P_a$ (one on each boundary component). Which exact triple we choose, will not play a role in what follows.

Given such a pair $(\Gamma,t)$, let us write
\[ g(\Gamma,t)\;\;\text{and}\;\;\sys_a(\Gamma,t) \]
for the genus and the systole of $S_a(\Gamma,t)$ respectively. We furthermore define
\[\diam_a(\Gamma,t) = \max_{v,w\in V}\left\{\dist_{S_a(\Gamma,t)} (\partial v,\partial w) \right\}. \]
where $\partial v\subset S_a(\Gamma,t)$ denotes the boundary of the copy of $P_a$ corresponding to a vertex $v\in V$.

Our first observation is the following:
\begin{lem}\label{lem_diama} Let $(\Gamma,t)$ be as above and $a\geq a_0$. Then there exists a constant $R>0$ independent of $(\Gamma,t)$ and $a$ such that
\[\frac{3a}{4}+\diam_a(\Gamma,t) \geq \log(\card{V}) - R \]
\end{lem}

\begin{proof} Take a point $x\in\partial v$ for some $v\in V$ and consider a different vertex $w\in V$. By the definition of $\diam_a(\Gamma,t)$, there exist points $p_v\in\partial v$ and $p_w\in\partial w$ so that
\[\dist_{S_a(\Gamma,t)}(p_v,p_w) \leq \diam_a(\Gamma,t). \]
From \eqref{eq_defa0} we get that
\[ \dist_{S_a(\Gamma,t)}(p_v,x) \leq \diam(P_a) \leq \frac{a}{2}+C, \]
for some constant $C>0$ independent of $a$. Furthermore there exists a midpoint $y\in\partial w$ (one of the black points in Figure \ref{pic_pants}) so that
\[ \dist_{S_a(\Gamma,t)}(p_w,y) \leq \frac{a}{4}.\]
There exists constants $A,D>0$ independent of $a$ so that $P_a$ contains a disk of area $A$ at distance $D$ from any midpoint on $\partial P_a$. These constants have been made explicit by Parlier in \cite{Par1}. For the current computation we will stick to $D$ and $A$.

Combining all the observations above, we obtain that in the disk $D_r(x)\subset S_a(\Gamma,t)$ of radius
\[r = \diam_a(\Gamma,t) + \frac{3a}{4} + C' \]
contains at least $\card{V}$ disjoint disks of area $A$. Here $C'$ is slightly larger than $D+C$ so as to include the entire disk in every pair of pants. As such
\[\cosh\left(\diam_a(\Gamma,t) + \frac{3a}{4} + C'\right) \geq \mathrm{area}(D_r(x))\geq \card{V} A. \]
Applying $\cosh^{-1}$ to both sides now gives the inequality.
\end{proof}

Finally, we need the following proposition:
\begin{prp}\label{prp_existence} Let $a\geq a_0$ and 
\[1<r<3/2-8s(a)/a.\]
Then there exists a pair $(\Gamma,t)$ consisting of a trivalent graph $\Gamma=(V,E)$ and $t\in \{-1,1\}^E$ so that every simple closed geodesic $\gamma$ on $S_a(\Gamma,t)$ that is not homotopic to a pants curve satisfies
\[\ell(\gamma) \geq  r\cdot a.\]
\end{prp}

\begin{proof} There exist trivalent graphs of arbitrarily high girth (see for instance \cite{ES}). So we can pick a trivalent graph $\Gamma$ so that its girth $h(\Gamma)$ satisfies
\[h(\Gamma) \geq 2\cdot \lceil r\cdot a/s(a)\rceil. \]
Choose any $t\in \{-1,1\}^E$ and consider a simple closed geodesic $\gamma$ on $S_a(\Gamma,t)$ of length $\ell(\gamma)<r\cdot a$. Given a pair of pants $v\in V$ so that $\gamma\cap v\neq\emptyset$, we have
\[\ell(\gamma\cap v) \geq s(a). \]
This is because $\gamma\cap v$ is not homotpic to any segment in $\partial v$ relative to $\partial v$.

Now pick any pair of pants $v_0\in V$ such that $v_0\cap \gamma \neq \emptyset$. The above tells us that the pairs of pants $v\in V$ that intersect $\gamma$ non-trivially form a subgraph of $D_{\lceil r\cdot a/s(a) \rceil} (v_0)$, the graph ball of radius $\lceil r\cdot a/s(a) \rceil$ around $v_0$. Because of our assumption on the girth of $\Gamma$, we have 
\[D_{\lceil r\cdot a/s(a) \rceil} (v_0) \simeq B_{\lceil r\cdot a/s(a) \rceil},\]
as graphs. Corollary \ref{cor_nonpantscurves} now tells us that $\gamma$ is necessarily a pants curve.
\end{proof}

For $a>0$, define
\[ r(a) = \frac{2\cdot \diam(P_a)}{a}.  \]
Recall that by the definition of $a_0$ \eqref{eq_defa0} we have
\begin{equation}\label{eq_r}
1\leq r(a) \leq \min\left\{\frac{3}{2}-8s(a)/a,1+C/a\right\},
\end{equation}
for all $a\geq a_0$ and some constant $C>0$ independent of $a$.

Given $a\geq a_0$, let us write
\[\mathcal{G}_a = \st{(\Gamma,t)}{\substack{\displaystyle{\text{Any simple closed geodesic  }\gamma\text{ on }S_a(\Gamma,t)\text{ that is}} \\ \displaystyle{\text{not a pants curve satisfies } \ell(\gamma) \geq r(a)\cdot a}}}\]
The proposition above tells us that this set is not empty. Note that
\[\sys_a(\Gamma,t) = a \]
for any $(\Gamma,t)\in\mathcal{G}_a$. This is because the systole of a closed hyperbolic surface is necessarily a simple closed geodesic.

The fact that $\mathcal{G}_a$ is non empty makes it possible to define $\Gamma_a,t_a$ to be any pair so that
\[g(\Gamma_a,t_a) = \min_{(\Gamma,t)\in\mathcal{G}_a}\{g(\Gamma,t)\} \]

We claim that this pair defines a surface with the desired properties:
\begin{thm}\label{thm_pantssystoles} There exists a constant $K>0$, independent of $a$ so that
\[\sys_a(\Gamma_a,t_a)\geq \frac{4}{7}\log(g(\Gamma_a,t_a)) - K, \]
for all $a\geq a_0$. In particular, $S_a(\Gamma_a,t_a)$ is a surface with large systoles and a systolic pants decomposition.
\end{thm}

\begin{proof} What we will actually prove is that
\[ \diam_a(\Gamma_a,t_a) \leq r(a) \cdot a. \]
Suppose this not the case. This means that we can find vertices $v,w\in V$ so that 
\[ \dist_{S_a(\Gamma_a,t_a)}(\partial v, \partial w) \geq r(a)\cdot a. \]
If we remove the corresponding pairs of pants, we obtain a pair $(\Gamma',t')$ where $\Gamma'=(V',E')$ is a graph with vertex degrees bounded by $3$ and $t'\in\{-1,1\}^{E'}$. This pair still corresponds to a surface $S_a(\Gamma',t')$, depicted in Figure \ref{pic_modifygraph} below:
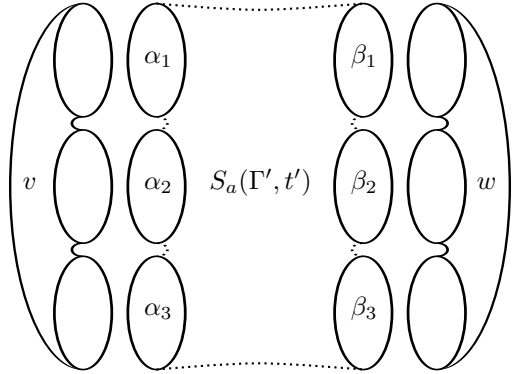
\begin{figure}[H]
\begin{center}
\begin{overpic}[scale=1]{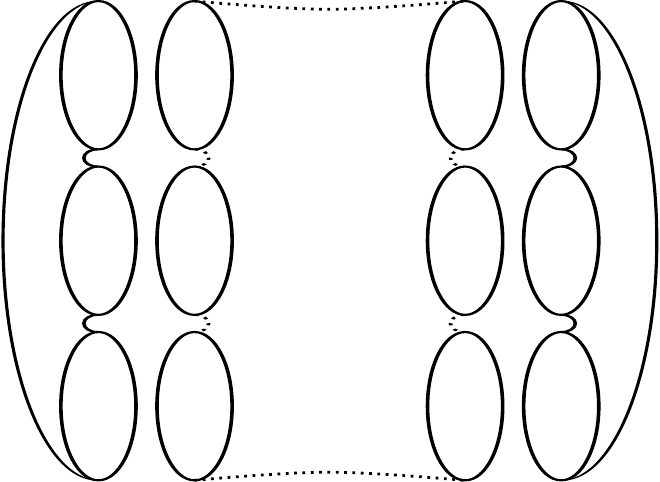} 
\put (3,36) {$v$}
\put (27,61) {$\alpha_1$}
\put (27,36) {$\alpha_2$}
\put (27,11) {$\alpha_3$}
\put (93,36) {$w$}
\put (68,61) {$\beta_1$}
\put (68,36) {$\beta_2$}
\put (68,11) {$\beta_3$}
\put (40,36) {$S_a(\Gamma',t')$}
\end{overpic}
\end{center}
\caption{Modifying the surface.}
\label{pic_modifygraph}
\end{figure}

We claim that when we glue the curve $\alpha_i$ to $\beta_i$ for $i=1,2,3$ in an admissible way, then the resulting surface still has no simple non-pants curves of length $\leq r(a)\cdot a$. Let us denote the corresponding graph and set of twists by $\Gamma''$ and $t''$ respectively.

First of all note that the surface $S_a(\Gamma',t')$ certainly has no such curves. So any simple non-pants curves of length $\leq r(a)\cdot a$ needs to consist of arcs between various $\alpha$- and $\beta$-curves.

The fact that $\dist_{S_a(\Gamma_a,t_a)}(\partial v,\partial w) \geq r(a)\cdot a$ tells us that any arc that runs from an $\alpha$-curve to a $\beta$-curve is too long. So if $S_a(\Gamma'',t'')$ has a non-pants curve of length $\leq r(a)\cdot a$, it consists of arcs running between $\alpha$-curves and arcs running between $\beta$-curves. 

Furthermore, such a curve necessarily contains at least two such arcs. Finally, we may assume that if any of these arcs has both endpoints on the same $\alpha$- or $\beta$-curve, it is not homotopic to a segment of the corresponding $\alpha$- or $\beta$- curve relative to its endpoints. Otherwise, we would be able to homotope the curve into $S_a(\Gamma',t')$, which is a case that we have already dealt with.

Consider a simple arc $\gamma$ running between $p,q\in\partial v$. On $S_a(\Gamma_a,t_a)$ we can use a simple arc between $p$ and $q$ in $v$ to complete $\gamma$ to a simple curve that is not homotopic to a pants curve. Because $S_a(\Gamma_a,t_a)$ has no non-pants curves of length $\leq r(a)\cdot a$ we have
\[\ell(\gamma) \geq r(a)\cdot a-\diam(P_a) = r(a)\cdot a/2. \]
But then any curve that contains more than two such arcs has length at least $r(a)\cdot a$. As such $S_a(\Gamma'',t'')$ indeed has no non-pants curves of length $\leq r(a)\cdot a$.

This means that $(\Gamma'',t'')\in\mathcal{G}_a$. But
\[ g(\Gamma'',t'') < g(\Gamma_a,t_a), \]
which contradicts our assumption on $(\Gamma_a,t_a)$. 

So indeed 
\[\diam_a(\Gamma_a,t_a) \leq r(a) \cdot a. \]
Filling in Lemma \ref{lem_diama} and using \eqref{eq_r} gives
\[\frac{7a}{4}+C\geq \left(\frac{3}{4}+r(a)\right)\cdot a \geq \log(\card{V}) = \log(2g(\Gamma_a,t_a)-2), \]
for some $C>0$ independent of $a$. Using that $a=\sys_a(\Gamma_a,t_a)$ and rearranging the terms now gives the result.
\end{proof}

Note that we have built these surfaces with admissible twists only. This is however not strictly necessary. The proof of Theorem \ref{thm_pantssystoles} nowhere uses admissibility of the twists. Admissible twists were in fact only convenient in the construction of a surface in which all non-pants curves are sufficiently long. As such, we could let go of the assumption and enlarge the set of surfaces $\mathcal{G}_a$.

\nocite{*}
\bibliographystyle{alpha}
\bibliography{systoles.bib}

\end{document}